\documentclass[11pt]{article}
\usepackage{amsmath}
\usepackage{amssymb}
\usepackage{amsthm}
\usepackage{bbm}
\usepackage{enumerate}
\usepackage[left=3.5cm,right=3.5cm,top=3cm,bottom=3cm]{geometry}
\usepackage{graphicx}
\usepackage{caption}
\usepackage{subcaption}
\usepackage[utf8]{inputenc}
\usepackage{pgfplots}

\pgfplotsset{compat=1.14}

\theoremstyle{definition}
\newtheorem{definition}{Definition}[section]

\newtheorem{example}[definition]{Example}
\newtheorem{assumption}[definition]{Assumption}

\theoremstyle{plain}
\newtheorem{proposition}[definition]{Proposition}
\newtheorem{theorem}[definition]{Theorem}
\newtheorem{lemma}[definition]{Lemma}
\newtheorem{corollary}[definition]{Corollary}

\theoremstyle{remark}
\newtheorem{remark}[definition]{Remark}

\numberwithin{equation}{section}

\newcommand{\E}{\mathbb{E}}

\renewcommand{\P}{\mathbb{P}}
\newcommand{\Q}{\mathbb{Q}}
\newcommand{\R}{\mathbb{R}}

\newcommand{\cA}{\mathcal{A}}

\newcommand{\cE}{\mathcal{E}}
\newcommand{\cF}{\mathcal{F}}

\newcommand{\cQ}{\mathcal{Q}}

\newcommand{\cV}{\mathcal{V}}

\newcommand{\cY}{\mathcal{Y}}

\newcommand{\rd}{\mathrm{d}}
\newcommand{\vp}{\varphi}
\newcommand{\pa}{\partial}

\newcommand{\hu}{\hat{u}}
\newcommand{\ox}{\overline{x}}
\newcommand{\ot}{\overline{t}}
\newcommand{\cVt}{\mathcal{V}^t_{\mathrm{loc}}}

\DeclareMathOperator*{\esssup}{ess\,sup}

\DeclareMathOperator*{\argmin}{arg\,min}

\begin{document}

\title{Parameter Uncertainty in the Kalman--Bucy Filter}
\author{Andrew L. Allan\thanks{Research supported by the Engineering and Physical Sciences Research Council [EP/L015811/1].}\hspace{2pt} and Samuel N. Cohen\thanks{Research supported by the Oxford-Man Institute for Quantitative Finance and the Oxford-Nie Financial Big Data Laboratory.}\\
Mathematical Institute, University of Oxford\\
andrew.allan\hspace{-0.6pt}{\fontfamily{ptm}\selectfont @}\hspace{-0.4pt}maths.ox.ac.uk\\
samuel.cohen\hspace{-0.6pt}{\fontfamily{ptm}\selectfont @}\hspace{-0.4pt}maths.ox.ac.uk}
\date{\today}
\maketitle

\vspace{-12pt}

\begin{abstract}
In standard treatments of stochastic filtering one first has to estimate the parameters of the model. Simply running the filter without considering the reliability of this estimate does not take into account this additional source of statistical uncertainty. We propose an approach to address this problem when working with the continuous time Kalman--Bucy filter, by making evaluations via a nonlinear expectation. We show how our approach may be reformulated as an optimal control problem, and proceed to analyse the corresponding value function. In particular we present a novel uniqueness result for the associated Hamilton--Jacobi--Bellman equation.

Keywords: Kalman--Bucy filter, parameter uncertainty, nonlinear expectation, optimal control, HJB equation.

MSC 2010: 62M20, 60G35, 93E11, 62F86, 49L25.
\end{abstract}

\section{Introduction}

In many applications one is interested in the behaviour of a process which cannot be observed directly, and must therefore rely on only partial observations in the presence of noise. Stochastic filtering is concerned with the problem of using such partial observations to infer the best estimate for the current value of the hidden process. One of the most famous results in this field is the Kalman--Bucy filter (Kalman \cite{Kalman1960}, Kalman and Bucy \cite{KalmanBucy1961}), which concerns the case where the underlying processes are Gaussian. Early applications of the Kalman--Bucy filter were to areospace engineering (see Cipra \cite{Cipra1993} or Grewal and Andrews \cite{GrewalAndrews2010}), and it has since found numerous applications in various fields from guidance and navigation to time series analysis and credit risk estimation. See also Crisan and Rozovskii \cite{CrisanRozovskii2011} for a comprehensive survey of nonlinear stochastic filtering.

The underlying equations in a filtering problem represent the dynamics of the (hidden) signal process, and the relationship between the signal and the observation processes. These equations typically involve various parameters so, naturally, in order to run the filter one must first estimate the values of these parameters. In standard treatments of stochastic filtering one simply runs the filter using an estimate of the parameters. However, this does not take into account the statistical uncertainty introduced by adopting this estimate. Particularly when there is limited available data, resulting in a lack of precision in the estimate, this should cast doubt as to the accuracy of the filter. This additional uncertainty evidently also has implications for decision making based on the outcome of the filter. By evaluating, say, the posterior variance of the signal, one can take into consideration the inherent risk associated with a decision based on the estimated values of the parameters, but this is not sufficient to account for the uncertainty associated with the error in the estimation of the parameters.

This objection is not new. Robustness in stochastic filtering has been considered in various different frameworks for both linear and nonlinear systems. Note that the term `robustness' has more than one meaning in the context of filtering. We are concerned here with robustness with respect to model uncertainty, rather than, say, continuity of the filter with respect to the observation path. Miller and Pankov \cite{MillerPankov2005}, Siemenikhin \cite{Siemenikhin2016}, Siemenikhin, Lebedev and Platonov \cite{SiemenikhinLebedevPlatonov2005} and Verd\'u and Poor \cite{VerduPoor1984} for instance consider known linear dynamics with uncertainty in the noise covariance matrices. However, as pointed out by Martin and Mintz \cite{MartinMintz1983}, the performance of a filter is typically more dependant on dynamic uncertainty than on uncertainties in either the signal or observation noise covariances.

A by now standard approach to robust estimation problems is to construct a minimax estimator, which generally speaking is designed to find the estimate which minimizes the maximum expected loss over a range of possible models, an idea that goes back at least as far as Wald \cite{Wald1945}. In addition to the aforementioned authors, this type of estimator has more recently been utilized by Borisov \cite{Borisov2008,Borisov2011}, who considers filtering of finite state Markov processes under uncertainty of the transition intensity and observation matrices. As pointed out in \cite{VerduPoor1984}, minimax estimators can be criticized for being too pessimistic. That is, since they are dependant on the specification of an often arbitrary uncertainty class, they are likely to consider very implausible models, but do not guarantee a satisfactory performance under the most statistically probable model. Moreover, the loss functions considered are invariably assumed to be quadratic which, although generally reasonable, is somewhat restrictive.

A more contemporary approach is the use of $H_\infty$ filters, which attempt to achieve a desired degree of noise attenuation by ensuring that the maximum energy gain from the noise source to the estimation error is bounded by a prescribed level. This type of filter is suitable in situations when an accurate system model or knowledge of the noise statistics is not available, although as noted in \cite{George2013}, since such estimators are designed to minimize the worst case energy gain, they generally sacrifice the average filter performance. For treatments of robust $H_\infty$ filters for linear systems see Lewis, Xie and Popa \cite{LewisXiePopa2008}, Xie, de Souza and Fu \cite{XiedeSouzaFu1991}, Yang and Ye \cite{YangYe2009}, Zhang, Xia and Shi \cite{ZhangXiaShi2009} and the references therein.

A related strategy is to combine such $H_\infty$ optimization with more classical least squares techniques, to form what are known as mixed $H_2/H_\infty$ filters. Variations on this type of filter are explored in Bernstein and Haddad \cite{BernsteinHaddad1989} and Khargonekar, Rotea and Baeyens \cite{KhargonekarRoteaBaeyens1996}, where in each case the goal is to minimize an $L^2$-type state estimation error criterion, subject to a prespecified $H_\infty$ constraint on this error. See also Aliyu and Boukas \cite{AliyuBoukas2009} for an application of this method to nonlinear filtering, or Chen and Zhou \cite{ChenZhou2002} for a game approach to mixed $H_2/H_\infty$ filter design.

Another recent technique is that of Neveux, Blanco and Thomas \cite{NeveuxBlancoThomas2007} and Zhou \cite{Zhou2010}, which penalises the sensitivity of estimation errors to parameter variations, by minimizing both the estimation error as well as its gradient with respect to the uncertain parameters.

The objective of the current work is to propose a new, substantially different approach to parameter uncertainty in stochastic filtering, specifically when working with the continuous time Kalman--Bucy filter. Based on ideas in Cohen \cite{Cohen2016discretefiltering,Cohen2017statisticaluncertainty}, we propose to make evaluations using a nonlinear expectation, represented in terms of a penalty function. This penalty represents how our uncertainty evolves in time, and is calculated by propagating our a priori uncertainty forward through time using the filter dynamics. The main result of this paper is to establish this penalty function as the unique solution of a Hamilton--Jacobi--Bellman (HJB) type partial differential equation (PDE). This means in particular that it can be calculated recursively through time alongside the standard filter.

The initial calibration of the penalty function may be based on classical statistics, but may be chosen according to the user's preference. By prescribing the initial penalty to be very large outside a particular range of parameter values, we essentially consider perturbations of the signal dynamics which are bounded by a prescribed level, recovering a setting comparable to that of $H_\infty$ filtering. Moreover, since our approach is very much an extension of the standard Kalman--Bucy filter, it combines elements of both $H_2$ and $H_\infty$ filtering techniques. The approach taken here requires one to solve a nonlinear PDE, which is computationally more expensive than typical $H_2$ and $H_\infty$ filtering methods, but allows us considerably more flexibility in terms of calibration and filter construction.

Our penalty function can be used to perform robust estimation by, for instance, constructing a minimax estimator with an arbitrary choice of loss function. Much of the aforementioned literature concerns estimators only of the state of the signal, whereas the approach taken here allows for estimation of \emph{any} random variable which depends on the current state of the signal. In particular, we can calculate robust estimates of nonlinear functionals of the signal. Moreover, we can compute any number of such estimates after the construction of a \emph{single} penalty function. That is, our penalty does not have to be recalculated for each choice of variable to be estimated. As we will see, our nonlinear expectation can also be used to define an interval estimate for any such random variable, thus providing uncertainty-robust error bounds, in analogy with classical confidence intervals.

In Section~\ref{SectionNLE}, after a brief recall of the Kalman--Bucy filter itself, we shall introduce the notion of nonlinear expectations, and see how they may be used to construct robust estimates in the context of model uncertainty. We will then see in Section~\ref{SectionReformulation} how the associated penalty may be formulated as the value function of a (deterministic) optimal control problem. This control problem turns out to be somewhat degenerate due to the highly nonlinear nature of the Hamiltonian, and the correspondingly unruly behaviour of the state trajectories. The second main contribution of this paper is to provide a thorough and rigorous treatment of this control problem and its relationship with the associated HJB equation, via techniques which are generalisable to similarly nonlinear control problems.

In Section~\ref{SectionProperties} we will establish important properties of the value function, as well as its counterpart obtained by restricting to uniformly bounded parameters. Our main result in this direction is uniqueness of solutions to the HJB equation, the proof of which will be demonstrated in Section~\ref{SectionUniqueness}. Although a thorough examination of the appropriate calibration and accuracy of our approach in comparsion to existing methods is beyond the scope of this paper, in Section~\ref{Sectionnumericalexample} we exhibit a numerical example to illustrate our results. The numerical scheme used to solve the HJB equation in this example is briefly outlined in the appendix.

\section{Estimation via nonlinear expectations}\label{SectionNLE}

Let us take an underlying filtered space $(\Omega,\cF,(\cF_t)_{t \geq 0})$, and suppose that a signal process $X$ and observation process $Y$ satisfy the following pair of linear equations
\begin{gather*}
\rd X_t = \alpha_tX_t\,\rd t + \sqrt{\beta_t}\,\rd B_t,\\
\rd Y_t = c_tX_t\,\rd t + \rd W_t,
\end{gather*}
with the initial conditions $Y_0 = 0$ and $X_0 \sim N(\mu_0,\sigma_0^2)$, for some $\mu_0 \in \R$, $\sigma_0 > 0$. Here, $B$ and $W$ are standard $\cF_t$-adapted Brownian motions which are assumed to be uncorrelated, and the parameters $\alpha \colon [0,\infty) \to \R$, $\beta \colon [0,\infty) \to [0,\infty)$ and $c \colon [0,\infty) \to \R$ are measurable, locally bounded (deterministic) functions. In this paper we shall suppose that we are uncertain of the parameters of the signal process, namely $\alpha$ and $\beta$. Uncertainty of $c$ is likely to be the subject of future work.

For simplicity of presentation we shall focus on the scalar case, where $X$ and $Y$ are one-dimensional, however all of our results extend analogously to the multivariate case (see Remark~\ref{RemarkMultivariate}).

The standard object to study in stochastic filtering is the conditional distribution process $\{\pi_t : t \geq 0\}$, given by $\pi_t\vp = \E[\vp(X_t)\,|\,\cY_t]$ for any bounded measurable function $\vp$, where $\cY_t$ is the (completed) $\sigma$-algebra generated by the observation process up to time $t$. We recall (from e.g.~Bain and Crisan \cite{BainCrisan2009} or Liptser and Shiryaev \cite{LiptserShiryaev2001}) that in this setting the conditional distribution $\pi_t$ of $X_t$ given $\cY_t$ is a normal distribution, and hence completely characterised by its mean $q_t = \E[X_t\,|\,\cY_t]$ and variance $R_t = \E\big[(X_t - q_t)^2\,\big|\,\cY_t\big]$ which satisfy the equations
\begin{gather}
\rd q_t = \alpha_tq_t\,\rd t + c_tR_t\,\rd I_t,\label{eq:KBmeandynamics}\\
\rd R_t = \big(\beta_t + 2\alpha_tR_t - c_t^2R_t^2\big)\rd t,\label{eq:KBvariancedynamics}\\
\rd I_t = \rd Y_t - c_tq_t\,\rd t,\nonumber
\end{gather}
where $I$ is known as the innovation process.

Naturally this distribution depends on the parameters $\alpha,\beta$, which in our setting are unknown. Typically one finds an estimate for the parameters, uses this estimate to compute the solution of \eqref{eq:KBmeandynamics} and \eqref{eq:KBvariancedynamics}, and can then calculate $\pi_t\vp = \E[\vp(X_t)\,|\,\cY_t]$ for a given choice of $\vp$.

Making evaluations via (standard) expectations immediately becomes problematic in the uncertain parameter setting, as uncertainty in the parameters corresponds naturally to uncertainty in the underlying probability measure. An advantage of nonlinear expectations is that they can be defined without reference to a specific measure.

Nonlinear expectations are closely linked to the theory of risk measures, which are tools used in finance to attempt to quantify the riskiness of a financial position, and are a natural object to study when dealing with model uncertainty. One typically defines a \emph{dynamic convex expectation} as a map $\cE(\,\cdot\,|\,\cY_t) \colon L^\infty(\cF) \to L^\infty(\cY_t)$ which satisfies the properties of monotonicity, translation equivariance, normalization and convexity. The convexity of the expectation corresponds to the notion that, in the context of decision making in the presence of uncertainty, diversification of a portfolio should not increase the associated risk.

A particularly nice kind of convex expectation are those which also satisfy the Fatou property. As proved separately by F\"ollmer and Schied \cite{FollmerSchied2004} and by Frittelli and Rosazza Gianin \cite{FrittelliRosazzaGianin2002} in the context of monetary risk measures, any dynamic convex expectation which satisfies the Fatou property admits a representation of the form
\begin{equation}\label{eq:NLErepresentation}
\cE(\xi\,|\,\cY_t) = \esssup_{\Q \in \cQ_t}\big\{\E_\Q[\xi\,|\,\cY_t] - \kappa_t(\Q)\big\},
\end{equation}
where $\cQ_t$ is a suitably chosen space of equivalent probability measures, and $\kappa_t$ is a $\cY_t$-measurable `penalty' function. For a more thorough discussion of the theory of nonlinear expectations and risk measures, see e.g.~Cohen \cite{Cohen2016discretefiltering}, Delbaen, Peng and Rosazza Gianin \cite{DelbaenPengRosazzaGianin2010} or F\"ollmer and Schied \cite{FollmerSchied2002,FollmerSchied2004}.

We shall now see how such nonlinear expectations can be used to estimate functions of the hidden process $X$ given our observations of $Y$, taking parameter uncertainty into account. The approach we pursue has not previously been explored in a continuous time setting, but is analogous to the dynamic generator, uncertain prior setting in Cohen \cite{Cohen2016discretefiltering}. The first thing to notice is that the probability measure which governs our system depends on the choices of $\alpha,\beta,\mu_0$ and $\sigma_0$. That is, any given choice of these parameters corresponds to a probability measure $\P^{\alpha,\beta,\mu_0,\sigma_0}$ which determines the behaviour of the processes $X$ and $Y$. We associate to the set of all possible choices of parameters the corresponding class of all admissible probability measures, and suppose that the `true' measure belongs to this class.

For the purpose of the following definition, we allow the parameters $\alpha,\beta$ to depend on the observation process $Y$. That is, we consider the wider class of all integrable $\cY_t$-predictable processes $(\alpha,\beta) \colon \Omega \times [0,\infty) \to \R \times [0,\infty)$ whose paths are locally bounded. Motivated by the representation in \eqref{eq:NLErepresentation}, for a given uncertainty aversion parameter $k_1 > 0$ and exponent $k_2 \geq 1$, we define for $\xi \in L^\infty(\cF)$,
\begin{equation}\label{eq:defnEpstoch}
\cE(\xi\,|\,\cY_t) = \esssup_{\alpha,\beta,\mu_0,\sigma_0}\bigg\{\E^{\alpha,\beta,\mu_0,\sigma_0}[\xi\,|\,\cY_t] - \bigg(\frac{1}{k_1}\bigg(\int_0^t \gamma(s,\alpha_s,\beta_s)\,\rd s + \kappa_0(\mu_0,\sigma_0^2)\bigg)\bigg)^{k_2}\bigg\},
\end{equation}
where we write $\E^{\alpha,\beta,\mu_0,\sigma_0}$ for the expectation corresponding to the measure $\P^{\alpha,\beta,\mu_0,\sigma_0}$. Here the essential supremum is taken over all possible parameters $\mu_0,\sigma_0$ for the initial distribution of the signal, and over the class of all $\cY_t$-predictable processes $\alpha,\beta$, as described above. Allowing the parameters $\alpha$ and $\beta$ to be drawn from the class of $\cY_t$-predictable processes (rather than just deterministic functions) in this definition is not strictly necessary. However, if the essential supremum is attained then the optimum choice of parameters will in general belong to this class, so it is natural to consider it here.

The first term inside the essential supremum is a version of the conditional expectation of $\xi$, corresponding to the measure associated with a particular choice of the parameters. We consider the value of this conditional expectation for all choices of the parameters, and then discount this value by a penalty associated with the corresponding parameters. This penalty is composed of the nonnegative functions $\gamma$ and $\kappa_0$, where $\gamma$ may be deterministic or $\cY_t$-predictable. These penalty functions represent our opinion of how plausible different values of the parameters are. That is, we penalise different choices of parameters according to how unreasonable we consider them to be. Here $\gamma$ and $\kappa_0$ may be chosen arbitrarily, though in practice may be based on a priori statistical analysis, such as via the negative log-likelihood ratio with respect to some reference parameters, as in \cite{Cohen2016discretefiltering,Cohen2017statisticaluncertainty}. The parameters $k_1,k_2$ are included for generality, but will make essentially no change to our analysis; see \cite{Cohen2016discretefiltering} or \cite{Cohen2017statisticaluncertainty} for a proper discussion of their role.

Assuming that the functions $\gamma$ and $\kappa_0$ achieve the value zero, that is there exists a choice of reference parameters which have zero associated penalty, it follows that $\cE(\cdot\,|\,\cY_t) \colon L^\infty(\cF) \to L^\infty(\cY_t)$ defines a dynamic convex expectation which satisfies the Fatou property.

Provided that the reference parameters are chosen reasonably well, the nonlinear expectation $\cE(\xi\,|\,\cY_t)$ will typically overestimate the true value of $\xi$. We may therefore think of $\cE(\xi\,|\,\cY_t)$ as an `upper' expectation of $\xi$, the corresponding `lower' expectation being given by $-\cE(-\xi\,|\,\cY_t)$. The interval estimate $\big[\hspace{-2pt}-\hspace{-1pt}\cE(-\xi\,|\,\cY_t),\, \cE(\xi\,|\,\cY_t)\big]$ may then be considered as a range of plausible values that the true value of $\xi$ could take. This is of course analogous to the notion of confidence/credible intervals from classical statistics, and in fact this connection can be made rigorous in the case when $k_2 \to \infty$, as discussed in \cite{Cohen2017statisticaluncertainty}.

One can also use the nonlinear expectation to obtain a robust point estimate of $\xi$, by finding the $\cY_t$-measurable random variable $\hat{\xi}$ which minimizes $\cE(\psi(\xi - \hat{\xi})\,|\,\cY_t)$ for a given loss function $\psi$. This defines a type of minimax estimator. However, note that the inclusion of the penalty function ensures that parameter choices which are considered to be very unreasonable will never be considered, thus avoiding the previously mentioned criticism of other minimax estimators of being too pessimistic. Moreover, whereas many such estimators assume that the loss function is quadratic, the function $\psi$ in our formulation may be chosen completely arbitrarily.

The focus of the current work is the application of the approach described above to our filtering problem. Correspondingly, we shall henceforth restrict our attention to evaluations of random variables of the form $\xi = \vp(X_t)$, where $X$ is the (unobserved) signal process and $\vp$ is an arbitrary bounded measurable function.

We would like to have a version of our nonlinear expectation which depends directly on a particular observation of $Y$. By the Doob--Dynkin lemma, the conditional expectation $\E^{\alpha,\beta,\mu_0,\sigma_0}[\vp(X_t)\,|\,\cY_t]$ may be considered as a function of $\{Y_s : 0\leq s\leq t\}$, and we will henceforth write
$$\E^{\alpha,\beta,\mu_0,\sigma_0}[\vp(X_t)\,|\,y] := \E^{\alpha,\beta,\mu_0,\sigma_0}[\vp(X_t)\,|\,\cY_t](y)$$
for a given realisation $y = \{y(s) : 0\leq s\leq t\}$ of $Y$ up to time $t$. As we wish to consider this expectation for a potentially uncountable collection of functions $\vp$, one should take care to make sure that they are defined simultaneously. This may be done for all continuous bounded $\vp$ by first considering a countable collection of functions and appealing to the Stone--Weierstrass theorem.

As all the processes involved are $\cY_t$-adapted, restricting to a given realisation $y$ of $Y$ gives the following deterministic version of \eqref{eq:defnEpstoch},
\begin{align}
&\cE(\vp(X_t)\,|\, y)\nonumber\\
&= \sup_{\alpha,\beta,\mu_0,\sigma_0}\bigg\{\E^{\alpha,\beta,\mu_0,\sigma_0}[\vp(X_t)\,|\,y] - \bigg(\frac{1}{k_1}\bigg(\int_0^t \gamma(s,\alpha_s,\beta_s)\,\rd s + \kappa_0(\mu_0,\sigma_0^2)\bigg)\bigg)^{k_2}\bigg\},\label{eq:NLEony}
\end{align}
where the supremum is taken over all initial parameters $\mu_0,\sigma_0$ and all (deterministic) measurable, locally bounded functions $\alpha,\beta$.

Once we have fixed an observation path $y$, the set of all possible conditional distributions of $X_t$ given $y$ is simply the set of all (nondegenerate) normal distributions, which we naturally parametrise by their mean $\mu \in \R$ and variance $\sigma^2 > 0$. Recalling that the conditional variance $R_t$ does not depend on the observation process $Y$ (as can be seen in \eqref{eq:KBvariancedynamics}), we shall write $(q(y)_t,R_t)$ for the parameters of the conditional distribution once we have fixed the path $y$. Although it is not obvious that these variables are well-defined for every $y$ and every choice of parameters, we will not dwell on such measure theoretic issues here; see Remark~\ref{Remarkdefns}.

Our next aim is to show how the variable $\cE(\vp(X_t)\,|\, y)$ is related to the following `value function'. Here and throughout, we adopt the convention that $\inf\emptyset = \infty$. We define
\begin{equation}\label{eq:defnkappa}
\kappa_t(\mu,\sigma^2\,|\,y) = \inf\bigg\{\int_0^t\gamma(s,\alpha_s,\beta_s)\,\rd s + \kappa_0(\mu_0,\sigma^2_0) \ \bigg| \ \genfrac{}{}{0pt}{}{\alpha,\beta,\mu_0,\sigma_0 \quad \text{such that}}{\big(q(y)_t,R_t\big) = (\mu,\sigma^2)}\bigg\},
\end{equation}
where the infimum is taken over all measurable, locally bounded functions $\alpha,\beta$ and all initial parameters $\mu_0,\sigma_0$ such that the corresponding mean $q$ and variance $R$ of the conditional distribution of $X_t$ given $\cY_t$, under the measure $\P^{\alpha,\beta,\mu_0,\sigma_0}$, evaluated on the observation path $y$, are equal to $\mu$ and $\sigma^2$ respectively. As the following proposition suggests, the function $\kappa_t$ may be viewed as the convex dual of our nonlinear expectation.

In the following we will write $\phi^{\mu,\sigma^2}$ for the probability density function of a $N(\mu,\sigma^2)$ distribution.

\begin{proposition}
For every observation $y$ of $Y$ up to time $t$, we have that
\begin{equation}\label{eq:NLErelatestokappa}
\cE(\vp(X_t)\,|\, y) = \sup_{(\mu,\sigma) \in \R \times (0,\infty)}\bigg\{\int_\R\vp(x)\phi^{\mu,\sigma^2}(x)\,\rd x - \bigg(\frac{1}{k_1}\kappa_t(\mu,\sigma^2\,|\,y)\bigg)^{k_2}\bigg\}.
\end{equation}
\end{proposition}

\begin{proof}
Let $\mu \in \R$, $\sigma \in (0,\infty)$, and take any $\alpha,\beta,\mu_0,\sigma_0$ such that $(q(y)_t,R_t) = (\mu,\sigma^2)$, i.e.~the conditional distribution of $X_t$ given $\cY_t$ under the measure $\P^{\alpha,\beta,\mu_0,\sigma_0}$, evaluated on the observation $y$, is a normal distribution with mean $\mu$ and variance $\sigma^2$. We then have that
\begin{align*}
\E&^{\alpha,\beta,\mu_0,\sigma_0}[\vp(X_t)\,|\,y] - \bigg(\frac{1}{k_1}\bigg(\int_0^t \gamma(s,\alpha_s,\beta_s)\,\rd s + \kappa_0(\mu_0,\sigma_0^2)\bigg)\bigg)^{k_2}\\
&= \int_\R \vp(x)\phi^{\mu,\sigma^2}(x)\,\rd x - \bigg(\frac{1}{k_1}\bigg(\int_0^t \gamma(s,\alpha_s,\beta_s)\,\rd s + \kappa_0(\mu_0,\sigma_0^2)\bigg)\bigg)^{k_2}.
\end{align*}
Taking a supremum over all such choices of $\alpha,\beta,\mu_0$ and $\sigma_0$, we obtain
\begin{align*}
\sup_{\genfrac{}{}{0pt}{}{\alpha,\beta,\mu_0,\sigma_0 \ \, \text{such that}}{(q(y)_t,R_t) = (\mu,\sigma^2)}}&\bigg\{\E^{\alpha,\beta,\mu_0,\sigma_0}[\vp(X_t)\,|\,y] - \bigg(\frac{1}{k_1}\bigg(\int_0^t \gamma(s,\alpha_s,\beta_s)\,\rd s + \kappa_0(\mu_0,\sigma_0^2)\bigg)\bigg)^{k_2}\bigg\}\\
&= \int_\R \vp(x)\phi^{\mu,\sigma^2}(x)\,\rd x - \bigg(\frac{1}{k_1}\kappa_t(\mu,\sigma^2\,|\,y)\bigg)^{k_2}.
\end{align*}
Taking a supremum over all $\mu \in \R$ and $\sigma > 0$ then gives
\begin{align*}
\sup_{\alpha,\beta,\mu_0,\sigma_0}&\bigg\{\E^{\alpha,\beta,\mu_0,\sigma_0}[\vp(X_t)\,|\,y] - \bigg(\frac{1}{k_1}\bigg(\int_0^t \gamma(s,\alpha_s,\beta_s)\,\rd s + \kappa_0(\mu_0,\sigma_0^2)\bigg)\bigg)^{k_2}\bigg\}\\
&= \sup_{(\mu,\sigma) \in \R \times (0,\infty)}\bigg\{\int_\R\vp(x)\phi^{\mu,\sigma^2}(x)\,\rd x - \bigg(\frac{1}{k_1}\kappa_t(\mu,\sigma^2\,|\,y)\bigg)^{k_2}\bigg\}
\end{align*}
and, by \eqref{eq:NLEony}, we have the result.
\end{proof}

It also follows from the previous Proposition that
\begin{align}
&-\cE(-\vp(X_t)\,|\, y)\nonumber\\
&= \inf_{\alpha,\beta,\mu_0,\sigma_0}\bigg\{\E^{\alpha,\beta,\mu_0,\sigma_0}[\vp(X_t)\,|\,y] + \bigg(\frac{1}{k_1}\bigg(\int_0^t \gamma(s,\alpha_s,\beta_s)\,\rd s + \kappa_0(\mu_0,\sigma_0^2)\bigg)\bigg)^{k_2}\bigg\}\nonumber\\
&= \inf_{(\mu,\sigma) \in \R \times (0,\infty)}\bigg\{\int_\R\vp(x)\phi^{\mu,\sigma^2}(x)\,\rd x + \bigg(\frac{1}{k_1}\kappa_t(\mu,\sigma^2\,|\,y)\bigg)^{k_2}\bigg\}.\label{eq:NLJrelatestokappa}
\end{align}

\section{Reformulation as an optimal control problem}\label{SectionReformulation}

If one could obtain an explicit expression, or approximate expression, for the function $\kappa_t(\cdot,\cdot\,|\,y)$, then it would be a relatively straightforward task to compute the supremum in \eqref{eq:NLErelatestokappa} or the infimum in \eqref{eq:NLJrelatestokappa} for any sufficiently regular function $\vp$. We have therefore reduced our problem to understanding the behaviour of the function $\kappa$.

Our fundamental observation is that the expression for $\kappa$ given in \eqref{eq:defnkappa} has the familiar form of the value function of an optimal control problem. In this view, we interpret $\gamma$ as the running cost, $\kappa_0$ as the initial cost, $(\alpha,\beta)$ as a control, and $(q(y),R)$ as the state trajectory corresponding to a given control.

In the following we write $\cA$ for the space of controls, that is, the set of measurable, locally bounded functions $(\alpha,\beta) \colon [0,\infty) \to \R \times [0,\infty)$, and we denote the domain of the state trajectories by
$$U := \R \times (0,\infty).$$
The other key ingredient for this control problem is the dynamics of the state trajectories $q$ and $R$, given by \eqref{eq:KBmeandynamics} and \eqref{eq:KBvariancedynamics}. To summarise, we currently have:

\vspace{10pt}
\noindent\textbf{Control problem for $\kappa$.} \emph{For every $t > 0$ and $(\mu,\sigma^2) \in U$, find
$$\kappa_t(\mu,\sigma^2\,|\,y) = \inf\bigg\{\int_0^t\gamma(s,\alpha_s,\beta_s)\,\rd s + \kappa_0(\mu_0,\sigma^2_0) \ \bigg| \ \genfrac{}{}{0pt}{}{\alpha,\beta,\mu_0,\sigma_0 \quad \text{such that}}{\big(q(y)_t,R_t\big) = (\mu,\sigma^2)}\bigg\},$$
where $q,R$ satisfy \eqref{eq:KBmeandynamics} and \eqref{eq:KBvariancedynamics}, and the infimum is taken over all $(\alpha,\beta) \in \cA$ and all $(\mu_0,\sigma_0^2) \in U$ such that $(q(y)_t,R_t) = (\mu,\sigma^2)$.}
\vspace{10pt}

Notice that there is no expectation in the expression for the `value function' above. This is therefore a \emph{pathwise} stochastic optimal control problem. However, we anticipate being able to transform the SDE \eqref{eq:KBmeandynamics} into an ODE with random coefficients. The problem may then be treated by restricting our attention to an arbitrary observation path $y$ and considering the resulting \emph{deterministic} control problem.

The natural parameters, in the sense of exponential families, for a $N(\mu,\sigma^2)$ distribution are $\big(\mu/\sigma^2,-1/2\sigma^2)$. This motivates us to calculate
\begin{gather}
\rd\bigg(\frac{q}{R}\bigg)_{\hspace{-3pt}t} = -\frac{q_t}{R_t}\bigg(\alpha_t + \frac{\beta_t}{R_t}\bigg)\rd t + c_t\,\rd Y_t,\label{eq:SDEforqoverR}\\
\frac{\rd}{\rd t}\bigg(\frac{1}{R_t}\bigg) = -\frac{\beta_t}{R_t^2} - \frac{2\alpha_t}{R_t} + c_t^2.\nonumber
\end{gather}

Notice that the stochastic integral in \eqref{eq:SDEforqoverR} no longer involves the variable $R$. Since we assume that the parameter $c$ is known, once we have fixed an observation path $y$, we will also be able to fix the value of this integral.

Recall (from e.g.~\cite{BainCrisan2009}) that the innovation process $I$ is a $\cY_t$-adapted Brownian motion under $\P^{\alpha,\beta,\mu_0,\sigma_0}$. It then follows from the Burkholder--Davis--Gundy inequality and the Kolmogorov continuity criterion (see e.g.~\cite{CohenElliott2015}), that for $\P^{\alpha,\beta,\mu_0,\sigma_0}$-almost any $\omega \in \Omega$, the function $t \mapsto \big(\int_0^tc_s\,\rd Y_s\big)(\omega)$ is locally H\"older continuous with any exponent strictly less than $1/2$.

Note that although the measures $\P^{\alpha,\beta,\mu_0,\sigma_0}$ are not equivalent on $\cF_t$, they are equivalent on $\cY_t$. Since the stochastic integral is adapted to $\cY_t$, we can choose a version of this integral, independently of our choice of parameters, which is H\"older continuous for every $\omega \in \Omega$. Using this version, let us now fix an $\omega \in \Omega$ and the corresponding observation path $y$ (so that $Y_\cdot(\omega) = y$), and then define
$$\eta_t := \bigg(\int_0^tc_s\,\rd Y_s\bigg)(\omega) \qquad \text{for} \quad\ t \geq 0.$$
We may then rewrite equation \eqref{eq:SDEforqoverR} as the ODE
$$\frac{\rd}{\rd t}\bigg(\frac{q(y)_t}{R_t} - \eta_t\bigg) = -\frac{q(y)_t}{R_t}\bigg(\alpha_t + \frac{\beta_t}{R_t}\bigg).$$

We are now ready to make our change of variables. We transform the state trajectories from $(q(y),R)$ to $w$ by the relation
$$w(t) = (w_1(t),w_2(t)) = \bigg(\frac{q(y)_t}{R_t} - \eta_t, \ \frac{1}{R_t}\bigg).$$
By the above, $w$ satisfies the following ODEs
\begin{gather}
\frac{\rd w_1}{\rd t}(t) = -\big(w_1(t) + \eta_t\big)\big(\alpha_t + \beta_tw_2(t)\big),\nonumber\\
\frac{\rd w_2}{\rd t}(t) = -\beta_tw_2(t)^2 - 2\alpha_tw_2(t) + c_t^2,\label{eq:w2ODE}
\end{gather}
which may be equivalently expressed as
\begin{equation}\label{eq:controlODE}
\frac{\rd w}{\rd s}(s) = f(w(s),s,\alpha_s,\beta_s) \qquad \text{for} \quad 0 < s < t,
\end{equation}
where the function $f$ is given by
\begin{equation}\label{eq:defnf}
f(z,t,a,b) = \bigg(\hspace{-4pt}\begin{array}{c}
-(z_1 + \eta_t)(a + bz_2)\\
-bz_2^2 - 2az_2 + c_t^2
\end{array}\hspace{-4pt}\bigg) \qquad \text{for} \quad\ z = (z_1,z_2) \in U.
\end{equation}

We also make the natural change of variables in the terminal values of $q$ and $R$. That is, for a fixed time $t$, we make the transformation from $(\mu,\sigma^2)$ to $x$ via
\begin{equation}\label{eq:changevariablesmusigmaxt}
x = (x_1,x_2) = \bigg(\frac{\mu}{\sigma^2} - \eta_t, \ \frac{1}{\sigma^2}\bigg),
\end{equation}
noting that the map $(\mu,\sigma) \mapsto x$ is a bijection from $U$ to itself. With this change of variables, we define a new function $v$ by
\begin{equation}\label{eq:changevariablesvkappa}
v(x,t) = \kappa_t(\mu,\sigma^2\,|\,y), \qquad \text{for} \quad x \in U,\ t \in [0,\infty),
\end{equation}
where we suppress the dependence on $y$ in the notation, as it will henceforth be considered to be fixed. We also define a new initial cost function $v_0 \colon U \to \R$, defined such that
$$v_0(z) = \kappa_0(\mu_0,\sigma_0^2), \qquad \text{whenever} \qquad z = \bigg(\frac{\mu_0}{\sigma_0^2},\frac{1}{\sigma_0^2}\bigg) \in U.$$
Recalling the definition of $\kappa$, given in \eqref{eq:defnkappa}, we observe that $v$ may be expressed as
$$v(x,t) = \inf\bigg\{\int_0^t\gamma(s,\alpha_s,\beta_s)\,\rd s + v_0(w(0)) \ \bigg| \ \genfrac{}{}{0pt}{}{(\alpha,\beta) \in \cA,\ w(0) \in U}{\text{such that}\quad w(t) = x}\bigg\},$$
where the infimum is taken over all controls $(\alpha,\beta)$ and initial values $w(0)$ such that the corresponding state trajectory $w(\cdot)$ satisfies $w(t) = x$.

\begin{remark}
The change of variables we have chosen is, to our knowledge, the simplest one in this setting which successfully reduces our SDE into an ODE with random coefficients, but it is by no means unique. In general this type of result can be achieved via Doss--Sussmann transformations. In particular, the transformation exhibited in Buckdahn and Ma \cite{BuckdahnMa2007} is also sufficient for our purposes, and results in an ODE with comparable complexity.
\end{remark}

To consider the function $v$ as a value function in a strictly classical sense, we should express it as an infimum over only the space of controls, without specifying in its definition that the initial point $w(0)$ should be in $U$. One should notice however that for a given terminal value $w(t) = x$, there may exist controls such that the corresponding state trajectory $w$ leaves the domain $U$. Moreover, since the state trajectories satisfy an ODE \eqref{eq:controlODE} which has quadratic growth (note the quadratic terms in \eqref{eq:defnf}), some choices of control may result in trajectories with an infinite speed of propagation, which in particular may `blow up' in a finite time.

\begin{example}
Suppose that $c \equiv 2$, and consider the control given by $\alpha \equiv 0$ and $\beta \equiv 1$. Recall the equation \eqref{eq:w2ODE} for the second component of $w$, which in this case becomes
$$\frac{\rd w_2}{\rd s}(s) = 4 - w_2(s)^2.$$
If $x_2 = w_2(t) = 1$, then the solution of this equation is given by $w_2(s) = (6 - 2e^{4(t-s)})/(3 + e^{4(t-s)})$, which leaves the domain $U$ when $s = t - \frac{1}{4}\log(3)$.

If instead $x_2 = w_2(t) = 3$, then the solution is given by $w_2(s) = (10 + 2e^{4(t-s)})/(5 - e^{4(t-s)})$, which `blows up' when $s = t - \frac{1}{4}\log(5)$.
\end{example}

The previous example demonstrates some undesirable behaviour of the state trajectories. However we can prevent these types of behaviour by assigning an infinite initial cost to such trajectories, as follows. Let us extend the domain of the state trajectories $w$ to the entirety of $\R^2$. We extend the function $f$ by setting
\begin{equation}\label{eq:defnfx2neg}
f(z,t,a,b) = \bigg(\hspace{-4pt}\begin{array}{c}
-(z_1 + \eta_t)a\\
c_t^2
\end{array}\hspace{-4pt}\bigg) \qquad \text{for} \quad\ z = (z_1,z_2) \in \R^2 \setminus U,
\end{equation}
noting that, provided $c$ is continuous (which we shall assume), the function $f \colon \R^2 \times [0,\infty) \times \R \times [0,\infty) \to \R$ is continuous in each of its variables. Provided that the resulting extended function is continuous, how we choose to define the function $f$ for $x \in \R^2 \setminus U$ is not important, since we will assign an infinite initial cost to any trajectory which leaves $U$. It will however be useful later to have $f$ explicitly defined in the whole of $\R^2$.

We suppose that the trajectories $w$ satisfy the ODE \eqref{eq:controlODE} with this extended function $f$. We also extend the domain of the initial cost function $v_0$ to $\R^2$ by defining
$$v_0(z) = \infty \qquad \text{for} \quad\ z \in \R^2 \setminus U.$$
Although we don't actually obtain an initial point for trajectories which `blow up' in a finite time, we may simply prescribe an infinite `initial' cost to all such trajectories.

For each terminal value $x \in U$, we consider solutions $w(\cdot)$ of \eqref{eq:controlODE} with the terminal condition
\begin{equation}\label{eq:terminalcondx}
w(t) = x \in U.
\end{equation}
We will sometimes write $w(\cdot) = w(\cdot\,;x,t,\alpha,\beta)$ to make explicit the dependence of $w(\cdot)$ on $x,t,\alpha$ and $\beta$.

We define a cost functional $J$, for $x \in U$, $t \geq 0$ and $(\alpha,\beta) \in \cA$, by
$$J(x,t;\alpha,\beta) = \int_0^t\gamma(s,\alpha_s,\beta_s)\,\rd s + v_0(w(0)),$$
where $w(0) = w(0;x,t,\alpha,\beta)$ is the initial value of the solution $w(\cdot)$ of \eqref{eq:controlODE} and \eqref{eq:terminalcondx}, and note that $v$ is then simply given by
\begin{equation}\label{eq:valuefuncdefnJ}
v(x,t) = \inf_{(\alpha,\beta) \in \cA}J(x,t;\alpha,\beta).
\end{equation}
Naturally, we have that $v(x,0) = v_0(x)$ for all $x \in U$. As noted above, the initial value $w(0)$ may not always lie in $U$, but such trajectories can be ignored when taking the infimum in \eqref{eq:valuefuncdefnJ} as we have assigned them an infinite cost.

To summarise, we now have:

\vspace{10pt}
\noindent\textbf{Control Problem for $v$.} \emph{For every $x \in U$ and $t > 0$, find
\begin{equation}\label{eq:valuefuncdefn}
v(x,t) = \inf_{(\alpha,\beta) \in \cA}J(x,t;\alpha,\beta),
\end{equation}
where
\begin{equation*}
J(x,t;\alpha,\beta) = \int_0^t\gamma(s,\alpha_s,\beta_s)\,\rd s + v_0(w(0)),
\end{equation*}
and $w(\cdot) = w(\cdot\,;x,t,\alpha,\beta)$ satisfies $w(t) = x$ and
\begin{equation*}
\frac{\rd w}{\rd s}(s) = f(w(s),s,\alpha_s,\beta_s) \qquad \text{for} \quad 0 < s < t,
\end{equation*}
where the function $f$ is given by \eqref{eq:defnf} and \eqref{eq:defnfx2neg}.}
\vspace{10pt}

We have derived a (deterministic) optimal control problem, whose value function $v$ determines the function $\kappa$ via a simple change of variables. We will henceforth refer to \eqref{eq:valuefuncdefn} as the definition of $v$.

\begin{remark}\label{Remarkdefns}
In Section~\ref{SubsectionRegularity} it will be shown in particular that $v$ is a continuous function of $x$ and $t$. Given this fact, if one does take the value function of the control problem described above as the definition of $v$, then one can define $\kappa$ by the change of variables given by \eqref{eq:changevariablesmusigmaxt} and \eqref{eq:changevariablesvkappa}, and continuity of $\kappa$ is then immediate. One can then read the supremum in \eqref{eq:NLErelatestokappa} as the \emph{definition} of the nonlinear expectation evaluated on the observation path $y$, for any measurable function $\vp$ such that the supremum takes a finite value. In this view, we can avoid any and all measure theoretic issues involved in defining all of the aforementioned objects.
\end{remark}

\section{Properties of the value function}\label{SectionProperties}

\subsection{Relationship with the HJB equation}

Having derived the optimal control problem described above, we now proceed to study the value function $v$ in some detail. In particular, we aim to identify it as the unique solution of the associated Hamilton--Jacobi--Bellman (HJB) equation:
\begin{equation}\label{eq:HJB}
\frac{\pa u}{\pa t}(x,t) + H\big(x,t,\nabla u(x,t)\big) = 0, \qquad \quad (x,t) \in U \times (0,\infty),
\end{equation}
where the Hamiltonian $H$ is defined by
\begin{equation}\label{eq:defnHamiltonian}
H(x,t,p) = \sup_{(a,b) \in \R \times [0,\infty)}\big\{f(x,t,a,b)\cdot p - \gamma(t,a,b)\big\}.
\end{equation}

Recall the function $f$ given by \eqref{eq:defnf} (and \eqref{eq:defnfx2neg}). Since the functions $c$ and $\eta$ are locally bounded, one can show that for every $T > 0$ there exists a constant $L_f > 0$ such that for all $x,z \in U$, $t \in [0,T]$ and $(a,b) \in \R \times [0,\infty)$, one has
\begin{equation}\label{eq:boundonf}
\big|f(x,t,a,b)\big| \leq L_f\big(1 + |a| + |a||x| + b|x| + b|x|^2\big)
\end{equation}
and
\begin{equation}\label{eq:loclipboundonf}
\big|f(x,t,a,b) - f(z,t,a,b)\big| \leq L_f\big(|a| + b + b|x| + b|z|\big)|x-z|.
\end{equation}
In particular we note that $f$ has linear growth in $a$ and $b$, locally uniformly in $x$ and $t$. That is, for any compact $K \subset U$ and any $T > 0$, there exists a constant $\tilde{L}_f > 0$ such that
\begin{equation}\label{eq:flingrowth}
\big|f(x,t,a,b)\big| \leq \tilde{L}_f\big(1 + |a| + b\big) \qquad \text{for all} \quad x \in K,\, t \in [0,T],\, (a,b) \in \R \times [0,\infty).
\end{equation}

Henceforth we shall always assume that the function $c$ is continuous. We also make the following assumptions on the functions $\gamma$ and $v_0$.

\begin{assumption}\label{assumptiongamma}
We suppose that the running cost function $\gamma \colon (0,\infty) \times \R \times (0,\infty) \to \R$ is continuous, nonnegative, and satisfies the following coercivity condition. We suppose that there exists $p > 1$ such that for any $T > 0$,
\begin{equation}\label{eq:gammasuperlin}
\inf_{0 \leq t \leq T}\frac{\gamma(t,a,b)}{|a|^p + b^p} \longrightarrow \infty \qquad \text{as} \qquad |a| + b \longrightarrow \infty.
\end{equation}
\end{assumption}
It does not matter, at this stage, whether $\gamma$ is also defined and finite for $b = 0$, or whether $\gamma$ approaches $+\infty$ as $b \to 0^+$.

\begin{definition}\label{defncV}
We shall denote by $\cV$ the class of functions $u \colon \R^2 \to \R \cup \{\infty\}$ which are finite-valued and bounded below in $U$, and satisfy
\begin{gather*}
u(x) \longrightarrow \infty \qquad \text{as} \qquad |x| \longrightarrow \infty\\
\text{and} \quad u(x) = u(x_1,x_2) \longrightarrow \infty \quad \text{as} \quad x_2 \longrightarrow 0^+, \quad \text{uniformly in}\ x_1.
\end{gather*}
\end{definition}

\begin{assumption}\label{assumptionv0}
As in the previous section, we assume that the initial cost function $v_0 \colon \R^2 \to \R \cup \{\infty\}$ takes the value $\infty$ on $\R^2 \setminus U$. We also assume that $v_0$ belongs to the class $\cV$, and is locally Lipschitz continuous on $U$.
\end{assumption}

\begin{lemma}\label{lemmavfinitevalued}
The value function $v$ is finite-valued and locally bounded on $U$.
\end{lemma}

\begin{proof}
For a given point $(x,t) = (x_1,x_2,t) \in U \times [0,\infty)$, consider the control $(\overline{\alpha},\overline{\beta})$ given by
$$\overline{\alpha}_s = \frac{c_s^2 - x_2^2}{2x_2} \qquad \text{and} \qquad \overline{\beta}_s = 1 \qquad \text{for} \quad\ s \geq 0.$$
Since $c$ is locally bounded and $\gamma$ is continuous, we see that $\int_0^t\gamma(s,\overline{\alpha}_s,\overline{\beta}_s)\,\rd s < \infty$.

Recall the ODE \eqref{eq:controlODE} satisfied by the state trajectories, given in this case by
$$\frac{\rd w_1}{\rd s}(s) = -\big(w_1(s) + \eta_s\big)\frac{c_s^2 + x_2^2}{2x_2} \qquad \text{and} \qquad \frac{\rd w_2}{\rd s}(s) = 0.$$
The solution of the first equation above has at most exponential growth, and in particular has a finite speed of propagation. It is therefore clear that the initial point $w(0)$ corresponding to this control lies in $U$. Since $v_0$ is finite-valued in $U$, we see that $J(x,t;\overline{\alpha},\overline{\beta}) < \infty$, and hence that $v(x,t) < \infty$.

Let $T > 0$ and let $K$ be a compact subset of $U$. Note that for $x \in K$, the value of $x_2$ is bounded away from zero, so that the control $(\overline{\alpha},\overline{\beta})$ defined above (which depends on the choice of $x$) is uniformly bounded on the compact time interval $[0,T]$. It is also clear that the initial point $w(0) = w(0;x,t,\overline{\alpha},\overline{\beta})$ is bounded in absolute value, and bounded away from the boundary of $U$ at $x_2 = 0$, uniformly for $x \in K$ and $t \in [0,T]$. Since $v_0$ is assumed to be continuous and hence locally bounded, we deduce that $\sup_{x \in K,\, t \in [0,T]}J(x,t;\overline{\alpha},\overline{\beta}) < \infty$, and hence that $v$ is bounded on $K \times [0,T]$.
\end{proof}

Although we are primarily interested in the value function $v$, it is also worth considering the `approximate value function', given by taking the infimum of the cost functional $J$ over a class of uniformly bounded controls. More precisely, for a given constant $M > 0$, we define
$$\cA^M = \big\{(\alpha,\beta) \in \cA\hspace{2pt} :\hspace{2pt} |\alpha_s| + \beta_s \leq M\hspace{5pt} \ \text{for all}\ \ s \geq 0\big\}$$
to be the set of controls which are bounded by $M$, and let
\begin{equation}\label{eq:defnvM}
v^M(x,t) = \inf_{(\alpha,\beta) \in \cA^M}J(x,t;\alpha,\beta) \qquad \text{for} \quad\ x \in U,\ t \geq 0.
\end{equation}

Although parts of our analysis become easier when we restrict to uniformly bounded controls, one should note that the approximate value function defined above is in some ways more awkward to deal with, as it is not finite-valued in the whole of $U$.

Recall the control $(\overline{\alpha},\overline{\beta})$ from the proof of Lemma~\ref{lemmavfinitevalued}. Note that $\overline{\alpha}_s \to +\infty$ as $x_2 \to 0^+$ (whenever $c_s \neq 0$). This is an indication of the fact that, as the terminal value $x$ approaches the boundary of $U$ at $x_2 = 0$, we require larger and larger controls to ensure that the corresponding state trajectory remains inside $U$. This can seen by noticing that for small values of $w_2(t)$, the dominant term in the ODE \eqref{eq:w2ODE} is $c_t^2$, which pushes the initial value of the solution towards (and indeed beyond) the boundary of $U$. Recalling that $x_2 = 1/\sigma^2$, we note that this behaviour corresponds to extremely large values of the conditional variance being considered to be very unreasonable.

Thus, for any $M > 0$, the approximate value function $v^M$ is infinite-valued in a nontrivial (and time-dependant) subset of the domain $U$. In the following we denote
$$Q := U \times (0,\infty)$$
and write $Q^M$ for the part of the domain where $v^M$ is finite-valued, that is
\begin{align*}
Q^M &:= \big\{(x,t) \in Q\ :\ v^M(x,t) < \infty\big\}\\
&= \big\{(x,t) \in Q\ :\ \exists (\alpha,\beta) \in \cA^M\ \ \text{such that}\ \ w(0;x,t,\alpha,\beta) \in U\big\}.
\end{align*}
Since $v_0 \in \cV$, it can be seen that $v^M(x,t)$ tends to infinity as $(x,t)$ approaches the boundary of the set $Q \setminus Q^M$.

For a fixed time $T > 0$, we will also write
$$Q_T := U \times (0,T) \qquad \text{and} \qquad Q^M_T := Q^M \cap Q_T.$$

\begin{remark}\label{remarkKTQMT}
It is not hard to see that $Q^M$ is an open subset of $Q$. Further, by inspecting the proof of Lemma~\ref{lemmavfinitevalued}, we can deduce that, for any $T > 0$ and any compact $K \subset U$, we will have that $K \times (0,T) \subset Q^M_T$ for all sufficiently large $M > 0$, so that in particular $\bigcup_{M > 0}Q^M_T = Q_T$. Since any given control $(\alpha,\beta) \in \cA$ is bounded on the compact time interval $[0,T]$, we also have that $v(x,t) = \inf_{M > 0}v^M(x,t)$ for all $(x,t) \in Q_T$.
\end{remark}

When considering uniformly bounded controls, it is natural to also consider the modified HJB equation:
\begin{equation}\label{eq:HJBM}
\frac{\pa u}{\pa t}(x,t) + H_M\big(x,t,\nabla u(x,t)\big) = 0,
\end{equation}
where $H_M$ is the modified Hamiltonian, given by
\begin{equation}\label{eq:defnHM}
H_M(x,t,p) = \sup\bigg\{f(x,t,a,b)\cdot p - \gamma(t,a,b)\ \bigg| \ \genfrac{}{}{0pt}{}{(a,b) \in \R \times [0,\infty)}{\text{such that} \ \ |a| + b \leq M}\bigg\}.
\end{equation}

\begin{lemma}
The Hamiltonian $H$ and modified Hamiltonian $H_M$ are continuous.
\end{lemma}

\begin{proof}
As the functions $f$ and $\gamma$ are continuous, and the supremum in \eqref{eq:defnHM} is taken over a compact set, it is straightforward to see that $H_M$ is continuous for any given $M$.

Let $T,R > 0$, let $K$ be a compact subset of $U$, and write $\overline{B}_R(0)$ for the closed ball in $\R^2$, centred at $0$ with radius $R$. It follows from the conditions \eqref{eq:flingrowth} and \eqref{eq:gammasuperlin} that
\begin{align*}
\sup_{\genfrac{}{}{0pt}{}{x \in K,\ 0 \leq t \leq T}{p \in \overline{B}_R(0)}}&\big\{f(x,t,a,b)\cdot p - \gamma(t,a,b)\big\}\\
&\leq \big(1 + |a| + b\big)\bigg(\tilde{L}_fR - \inf_{0 \leq t \leq T}\frac{\gamma(t,a,b)}{1 + |a| + b}\bigg)\, \longrightarrow\, -\infty
\end{align*}
as $|a| + b \to \infty$. We therefore deduce the existence of an $M > 0$ such that $H(x,t,p) = H_M(x,t,p)$ for all $(x,t,p) \in K \times [0,T] \times \overline{B}_R(0)$. Continuity of $H$ then clearly follows from the continuity of $H_M$.
\end{proof}

The value function $v$ and approximate value function $v^M$ satisfy the following Dynamic Programming Principle (DPP).

\begin{proposition}\label{DPP}
For any $x,t$ and $0 \leq r \leq t$, we have
$$v(x,t) = \inf_{(\alpha,\beta) \in \cA}\bigg\{\int_r^t\gamma(s,\alpha_s,\beta_s)\,\rd s + v(w(r),r)\bigg\},$$
where $w(\hspace{1pt}\cdot\hspace{1pt}) = w(\,\cdot\,;x,t,\alpha,\beta)$.
\end{proposition}

The statement of the DPP for the approximate value function $v^M$ is identical to the above, except that one only considers controls $(\alpha,\beta) \in \cA^M$. The proof of this result is standard; see e.g.~Yong and Zhou \cite{YongZhou1999}.

As noted above, the approximate value function $v^M$ is only finite-valued in the subdomain $Q^M$. However, we still wish to consider $v^M$ as being a supersolution of the HJB equation. We therefore consider viscosity supersolutions in the following sense.

\begin{definition}\label{Defninfinitesupersolns}
We say that a function $\hu \colon Q \to \R \cup \{\infty\}$ is a \emph{viscosity supersolution} of a given equation if the set $Q^{\hu} := \{(x,t) \in Q\, :\, \hu(x,t) < \infty\}$ is open, $\hu$ satisfies the usual definition of viscosity supersolution on $Q^{\hu}$, and $\hu$ is `continuous' in the sense that, as well as being continuous on $Q^{\hu}$, we also have that $\hu(x,t)$ tends to infinity as $(x,t)$ approaches the boundary of the set $Q \setminus Q^{\hu}$.
\end{definition}

We can use Proposition~\ref{DPP} to obtain the following result, whose proof is adapted from \cite{YongZhou1999}.

\begin{proposition}\label{valuefuncvissolnHJB}
The value function $v$ is a viscosity solution of the HJB equation \eqref{eq:HJB}, and the approximate value function $v^M$ is a viscosity supersolution of the modified HJB equation \eqref{eq:HJBM}.
\end{proposition}

\begin{proof}
We first note that $v$ is continuous on $Q$, and $v^M$ is continuous on $Q^M$, which follows from Theorem~\ref{valuefncloclip} below.

Suppose $\vp \in C^1(Q)$ is such that $v - \vp$ has a local maximum at a point $(x,t) \in Q$. Fix $(a,b) \in \R \times [0,\infty)$, and define a control $(\alpha,\beta) \in \cA$ by $(\alpha_t,\beta_t) \equiv (a,b)$, and let $w(\cdot) = w(\cdot\,;x,t,\alpha,\beta)$ be the corresponding state trajectory. Then, by the DPP (Proposition~\ref{DPP}), for $0 \leq r < t$ with $t - r$ sufficiently small, we have
\begin{align*}
0 &\leq \frac{v(x,t) - \vp(x,t) - v(w(r),r) + \vp(w(r),r)}{t - r}\\
&\leq \frac{1}{t - r}\bigg(\int_r^t\gamma(s,a,b)\,\rd s - \vp(x,t) + \vp(w(r),r)\bigg)\\
&\longrightarrow\, \gamma(t,a,b) - \frac{\pa\vp}{\pa t}(x,t) - f(x,t,a,b)\cdot\nabla\vp(x,t) \qquad \text{as} \quad r \to t.
\end{align*}
It follows that
$$\frac{\pa\vp}{\pa t}(x,t) + H\big(x,t,\nabla\vp(x,t)\big) \leq 0,$$
so $v$ is a viscosity subsolution of \eqref{eq:HJB}.

Fix $T > 0$ and suppose that $\vp \in C^1(Q^M_T)$ is such that $v^M - \vp$ has a local minimum at some point $(x,t) \in Q^M_T$. Let $\varepsilon > 0$. By the DPP, for any $0 \leq r < t$ with $t - r$ sufficiently small, we can find a control $(\alpha^r,\beta^r) \in \cA^M$ such that, writing $w(\cdot) = w(\cdot\,;x,t,\alpha^r,\beta^r)$, we have
\begin{align*}
0 &\geq v^M(x,t) - \vp(x,t) - v^M(w(r),r) + \vp(w(r),r) \\
&\geq -\varepsilon(t - r) + \int_r^t\gamma(s,\alpha^r_s,\beta^r_s)\,\rd s - \vp(x,t) + \vp(w(r),r).
\end{align*}
Rearranging, it follows that
\begin{align*}
-\varepsilon &\leq \frac{1}{t - r}\int_r^t\bigg(\frac{\pa\vp}{\pa t}(w(s),s) + f(w(s),s,\alpha^r_s,\beta^r_s)\cdot\nabla\vp(w(s),s) - \gamma(s,\alpha^r_s,\beta^r_s)\bigg)\rd s\\
&\leq \frac{1}{t - r}\int_r^t\bigg(\frac{\pa\vp}{\pa t}(w(s),s) + H_M\big(w(s),s,\nabla\vp(w(s),s)\big)\bigg)\rd s\\
&\longrightarrow\, \frac{\pa\vp}{\pa t}(x,t) + H_M\big(x,t,\nabla\vp(x,t)\big) \qquad \text{as} \quad r \to t.
\end{align*}
One should note that the function $w(\cdot)$ in the above depends on $r$, so care should be taken when taking the limit in the last line. The convergence is justified by the fact that the controls $\{(\alpha^r,\beta^r)\}_{0 \leq r < t}$ are uniformly bounded (by $M$), which implies that the functions $\{w(\cdot\,;x,t,\alpha^r,\beta^r)\}_{0 \leq r < t}$ are continuous at the point $(x,t)$, uniformly in $r$.

Taking $\varepsilon \to 0$, we deduce that $v^M$ is a viscosity supersolution of \eqref{eq:HJBM} in $Q^M_T$, and hence that it is a viscosity supersolution of \eqref{eq:HJBM} in $Q_T$, in the sense of Definition~\ref{Defninfinitesupersolns}.

Since $H_M \leq H$, we see that $v^M$ is also a viscosity supersolution of \eqref{eq:HJB} in $Q^M_T$. As noted in Remark~\ref{remarkKTQMT}, we have that $\bigcup_{M > 0}Q^M_T = Q_T$, and $v(x,t) = \inf_{M > 0}v^M(x,t)$ for all $(x,t) \in Q_T$, from which it follows that $v$ is also a viscosity supersolution of \eqref{eq:HJB} in $Q_T$. As $T > 0$ was arbitrary, we have the result.
\end{proof}

\subsection{Regularity}\label{SubsectionRegularity}

Recall Assumptions~\ref{assumptiongamma} and \ref{assumptionv0}, in particular the assumption that the initial cost function $v_0$ is locally Lipschitz continuous on $U$. Under these assumptions we have the following.

\begin{theorem}\label{valuefncloclip}
The value function $v$ is locally Lipschitz continuous.
\end{theorem}

This result may be proved by adapting arguments from Section~2 of Bardi and Da Lio \cite{BardiDaLio1997}, though care is needed to account for the nonlinearity in the function $f$, and for the fact that our spatial domain is a half-plane. In particular, the use of Young's inequality is the reason for our insistence on $p > 1$ in Assumption~\ref{assumptiongamma}. Moreover, the following key estimate is required.

\begin{lemma}\label{lemmakeyest}
Let $T > 0$. There exists a constant $C > 0$ such that for any control $(\alpha,\beta)$ and any terminal pair $(x,t) \in U \times [0,T]$, the corresponding state trajectory $w(\cdot) = w(\cdot\,;x,t,\alpha,\beta)$ satisfies
\begin{equation}\label{eq:logbound0r}
\log\big(1 + |x|^2\big) \leq \log\big(1 + |w(0)|^2\big) + C\int_0^t\big(1 + |\alpha_s| + \beta_s\big)\,\rd s.
\end{equation}
\end{lemma}

\begin{proof}
Using the fact that the functions $\eta$ and $c$ are locally bounded, we have
\begin{align*}
\frac{1}{2}\frac{\rd}{\rd s}&|w(s)|^2 = w(s) \cdot f(w(s),s,\alpha_s,\beta_s)\\
&= -w_1(s)\big(w_1(s) + \eta_s\big)\big(\alpha_s + \beta_sw_2(s)\big) -\beta_sw_2(s)^3 - 2\alpha_sw_2(s)^2 + c_s^2w_2(s)\\
&\leq C\big(1 + |\alpha_s| + \beta_s\big)\big(1 + |w(s)|^2\big),
\end{align*}
where crucially we notice that, since $w_2(s) > 0$ and $\beta_s \geq 0$, we are able to drop the highest order terms in $w$. Rearranging and integrating this expression, we obtain \eqref{eq:logbound0r}.
\end{proof}

The previous lemma implies a type of (partial) finite speed of propagation property. That is, although the state trajectories $w$ may have an infinite speed of propagation `toward infinity', they can only return `toward zero' with a finite speed.

One can also show that the approximate value function $v^M$ is locally Lipschitz continuous on $Q^M$, but we will not make use of this result.

An immediate consequence of Theorem~\ref{valuefncloclip} is that the value function $v$ belongs to the Sobolev space $W^{1,\infty}_{\mathrm{loc}}(Q)$. In particular, $v$ has a weak derivative, has a classical derivative almost everywhere, and these two notions of derivative coincide almost everywhere (see Evans \cite{Evans2010}). Moreover, $v$ satisfies the HJB equation \eqref{eq:HJB} in the classical sense at each point where its classical derivative exists.

Recall our assumption that the initial cost function $v_0$ belongs to the class $\cV$ (defined in Definition~\ref{defncV}). As the last result of this section, we show that the value function $v$ also satisfies a similar type of `weak coercivity', which one may also think of as an `explosive boundary condition'.

\begin{definition}\label{defnVtloc}
We shall denote by $\cVt$ the class of functions $u \colon Q \to \R$ which are bounded below, and satisfy, for every $T > 0$,
\begin{equation}\label{eq:valuefuncasymp1}
\inf_{0 \leq t \leq T}u(x,t) \longrightarrow \infty \qquad \text{as} \quad\ |x| \longrightarrow \infty,
\end{equation}
and
\begin{equation}\label{eq:valuefuncasymp2}
\inf_{x_1 \in \R,\ 0 \leq t \leq T}u(x_1,x_2,t) \longrightarrow \infty \qquad \text{as} \quad\ x_2 \longrightarrow 0^+.
\end{equation}
\end{definition}

\begin{lemma}\label{lemmaasympvaluefunc}
$v \in \cVt$.
\end{lemma}

\begin{proof}
Recall the estimate \eqref{eq:logbound0r} from Lemma~\ref{lemmakeyest}, which holds for any choice of control $(\alpha,\beta)$. Using \eqref{eq:gammasuperlin}, we obtain
\begin{equation}\label{eq:logbound1}
\log\big(1 + |x|^2\big) \leq \log\big(1 + |w(0)|^2\big) + C\int_0^t\big(1 + \gamma(s,\alpha_s,\beta_s)\big)\,\rd s,
\end{equation}
with a possibly different constant $C$. Recalling the definition of the value function,
\begin{equation*}
v(x,t) = \inf_{(\alpha,\beta) \in \cA}\bigg\{\int_0^t\gamma(s,\alpha_s,\beta_s)\,\rd s + v_0\big(w(0;x,t,\alpha,\beta)\big)\bigg\},
\end{equation*}
and the fact that $v_0 \in \cV$, we deduce that \eqref{eq:valuefuncasymp1} holds for $v$.

The second component of a given state trajectory satisfies
\begin{align*}
\frac{\rd w_2}{\rd s}(s) &= -\beta_sw_2(s)^2 - 2\alpha_sw_2(s) + c_s^2\\
&\geq -2\big(|\alpha_s| + \beta_s\big)w_2(s)\big(1 + w_2(s)\big),
\end{align*}
from which, using \eqref{eq:gammasuperlin} again, we obtain
\begin{equation}\label{eq:logbound2}
\log\bigg(\frac{w_2(0)}{1 + w_2(0)}\bigg) - C\int_0^t\big(1 + \gamma(s,\alpha_s,\beta_s)\big)\,\rd s \leq \log\bigg(\frac{x_2}{1 + x_2}\bigg)
\end{equation}
for some constant $C > 0$. From the definition of the value function and fact that $v_0 \in \cV$, we similarly deduce that \eqref{eq:valuefuncasymp2} holds for $v$.
\end{proof}

\section{Uniqueness for the HJB equation}\label{SectionUniqueness}

In the previous section we saw that the value function $v$ is a viscosity solution of the HJB equation \eqref{eq:HJB}. We now turn our attention to the question of whether this solution is unique.

There is a substantial amount of existing literature on uniqueness results for solutions of first-order HJB equations. The standard approach to obtaining such results is via a `doubling the variables' type argument, which provides what is essentially a maximum principle for viscosity solutions of certain PDEs (see Crandall, Ishii and Lions \cite{CrandallIshiiLions1992}). However, typical treatments of such equations on an unbounded domain (e.g.~in Barles \cite{AchdouBarlesLitvinovIshii2013} or Evans \cite{Evans2010}) assume that the solutions are bounded and uniformly continuous, whereas our solution $v$ exhibits explosive behaviour near the boundary of $U$. In Yong and Zhou \cite{YongZhou1999} solutions are only assumed to be continuous, but the Hamiltonian is assumed to grow at most linearly in the spatial variable $x$. In our case the state trajectories obey an ODE \eqref{eq:controlODE} which is quadratic in $x$, which means that our Hamiltonian also has quadratic growth in $x$. Moreover, although the function $\eta$ is locally H\"older continuous, it is not locally Lipschitz, as would be required in order to apply the sub/super-optimality principle (see Lions and Souganidis \cite{LionsSouganidis1985}).

On the other hand, when working on a bounded domain where the comparison is assumed to hold on the parabolic boundary of the domain, standard comparison results are less restrictive in their assumptions on the Hamiltonian. Our strategy here is therefore to exploit properties of the value function in order to be able to restrict attention to a bounded subdomain, and then apply a standard comparison theorem from Barles \cite{AchdouBarlesLitvinovIshii2013}.

The rest of this section is dedicated to the proof of Theorem~\ref{TheoremUniqueness} below, which we shall prove under the following additional mild assumption.

\begin{assumption}\label{assumptiongammafinite}
We assume that $\gamma(t,0,0) < \infty$ for all $t > 0$.
\end{assumption}

To the best of our knowledge our technique is new, and is applicable to more general nonlinear optimal control problems. The key is the fact that the value function $v$ depends only on local values of the initial cost function $v_0$. We can therefore manipulate the asymptotic behaviour of $v_0$, and hence that of $v$, without changing the local behaviour of $v$. This is the motivation for Propositions~\ref{PropvissubcompHJBM} and \ref{PropvissupercompHJB} below. We will then be able to deduce the following result.

\begin{theorem}\label{TheoremUniqueness}
The value function $v$ is the unique viscosity solution of the HJB equation \eqref{eq:HJB} in the class $\cVt$ which satisfies the initial condition $v(x,0) = v_0(x)$ for all $x \in U$.
\end{theorem}

Note that, rather than any sort of integrability or regularity condition in the usual sense, the additional condition we impose to establish uniqueness for our equation is actually an `explosive boundary condition'. Further, notice that although we insist that the relevant solutions converge to infinity as we approach the boundary of $U$, we make no restriction whatsoever on the \emph{speed} of this convergence.

In the following, by a viscosity (sub/super)solution we will always mean a continuous viscosity (sub/super)solution. As in the previous section, we consider supersolutions in the sense of Definition~\ref{Defninfinitesupersolns}. In particular, we allow viscosity supersolutions to be infinite-valued, provided that they are still `continuous' in the sense that they tend to infinity as they approach the part of the domain where they are infinite-valued.

It will be helpful to first restrict ourselves to the case of uniformly bounded controls, and consider solutions of the modified HJB equation \eqref{eq:HJBM}. From the bound \eqref{eq:loclipboundonf} on the function $f$, we easily deduce that for every $T > 0$ there exists a constant $L_M > 0$ such that
\begin{equation}\label{eq:xlipboundonHM}
\big|H_M(x,t,p) - H_M(z,t,p)\big| \leq L_M\big(1 + |x| + |z|\big)|x - z||p|
\end{equation}
for all $x,z \in U$, $t \in [0,T]$ and $p \in \R^2$.

\begin{proposition}\label{PropvissubcompHJBM}
Fix an $M > 0$, and let $u$ be a viscosity subsolution of the modified HJB equation \eqref{eq:HJBM} in $Q_T$, for some $T > 0$. Suppose that $u(x,0) \leq v_0(x)$ for all $x \in U$. Then $u \leq v^M$ in $Q_T$.
\end{proposition}

\begin{proof}
Fix an arbitrary compact set $K \subset U$, and choose $T_0 \in (0,T]$ sufficiently small such that $K \times (0,T_0) \subset Q^M$. Let us recall the definition of the approximate value function:
\begin{equation*}
v^M(x,t) = \inf_{(\alpha,\beta) \in \cA^M}\bigg\{\int_0^t\gamma(s,\alpha_s,\beta_s)\,\rd s + v_0\big(w(0;x,t,\alpha,\beta)\big)\bigg\}.
\end{equation*}
As $v^M$ is continuous on $Q^M$, it is bounded on $K \times (0,T_0)$. Since $\gamma$ and $v_0$ are bounded below, it follows that there exists a constant $C_0 > 0$ such that, when considering the value of $v^M(x,t)$ for $x \in K$ and $t \in (0,T_0)$, it is sufficient to consider only those controls $(\alpha,\beta)$ which satisfy
\begin{equation}\label{eq:boundonv0}
v_0\big(w(0;x,t,\alpha,\beta)\big) < C_0.
\end{equation}

For $\delta > 0$, define the set $U_\delta := \{x = (x_1,x_2) \in U\, :\, |x| < 1/\delta,\ x_2 > \delta\}$. Since $v_0 \in \cV$ (recall Definition~\ref{defncV}), there exists a $\delta > 0$ sufficiently small such that $K \subset U_\delta$, and such that
\begin{equation*}
v_0(x) \geq C_0 \quad \text{for all} \quad x \in U \setminus U_\delta.
\end{equation*}
For this $\delta$, let $\hat{V}_0$ be a continuous function belonging to the class $\cV$, such that
\begin{itemize}
\item $\hat{V}_0(x) = v_0(x)$ \ for \ $x \in U_\delta$,
\item $\hat{V}_0(x) \geq C_0$ \ for \ $x \in U \setminus U_\delta$,
\item $\hat{V}_0(x) = \infty$ \ for \ $x \in \R^2 \setminus U$.
\end{itemize}
We now define
\begin{equation}
\hat{V}(x,t) := \inf_{(\alpha,\beta) \in \cA^M}\bigg\{\int_0^t\gamma(s,\alpha_s,\beta_s)\,\rd s + \hat{V}_0\big(w(0;x,t,\alpha,\beta)\big)\bigg\}
\end{equation}
for $x \in U$ and $t \geq 0$, where, just as when we considered the value function $v$ in Section~\ref{SectionReformulation}, we associate an infinite `initial' cost to state trajectories which `blow up' after a finite time.

The function $\hat{V}$ is defined in the same way as the approximate value function $v^M$, except with a different initial cost function, namely $\hat{V}_0$ instead of $v_0$. We note that $\hat{V}$ is also finite-valued precisely in the subdomain $Q^M$. Further, the result of Proposition~\ref{valuefuncvissolnHJB} applies equally well to the function $\hat{V}$, which is hence a viscosity supersolution of the modified HJB equation \eqref{eq:HJBM}, with initial condition $\hat{V}(x,0) = \hat{V}_0(x)$ for all $x \in U$.

Since $v_0 \geq C_0$ and $\hat{V}_0 \geq C_0$ in $\R^2 \setminus U_\delta$, it follows from \eqref{eq:boundonv0} that when considering the values of $v^M(x,t)$ and $\hat{V}(x,t)$ for $x \in K$ and $t \in (0,T_0)$, we may ignore all controls $(\alpha,\beta)$ such that the corresponding initial point $w(0;x,t,\alpha,\beta)$ lies in $\R^2 \setminus U_\delta$. Since, $v_0$ and $\hat{V}_0$ are equal in $U_\delta$, we deduce that
\begin{equation}\label{eq:vMequalhatVonK}
v^M(x,t) = \hat{V}(x,t) \qquad \text{for all} \quad\ x \in K,\ t \in (0,T_0).
\end{equation}

Recall the estimates \eqref{eq:logbound1} and \eqref{eq:logbound2} from the proof of Lemma~\ref{lemmaasympvaluefunc}, which in particular hold for all controls $(\alpha,\beta) \in \cA^M$. Since $\hat{V}_0 \in \cV$, we can argue exactly as in Lemma~\ref{lemmaasympvaluefunc} that $\hat{V} \in \cVt$. Moreover, since $u$ is continuous and hence locally bounded, and since we are only considering controls which are uniformly bounded\footnote{The assumption of uniformly bounded controls is essential here. This argument could not therefore be applied directly to the value function $v$.} (by $M$), we see from \eqref{eq:logbound1} and \eqref{eq:logbound2} that by choosing $\hat{V}_0$ to grow sufficiently quickly as $|x| \to \infty$ and as $x_2 \to 0$, we can ensure that $\hat{V}_0(x) \geq v_0(x)$ for all $x \in U$, and that
$$\inf_{0 \leq t \leq T_0}\hat{V}(x,t) \geq \sup_{0 \leq t \leq T_0}u(x,t)$$
for all sufficiently large values of $|x|$, and for all sufficiently small values of $x_2$. Thus, there exists a bounded open subset $E \Subset U$, such that $K \subset E$, and such that $\hat{V} \geq u$ on $\pa E \times [0,T_0]$.

Since $\hat{V}(x,0) = \hat{V}_0(x) \geq v_0(x) \geq u(x,0)$ for all $x \in U$, we see that $\hat{V} \geq u$ on the parabolic boundary of $E \times (0,T_0)$. Since \eqref{eq:xlipboundonHM} holds, we may apply Theorem~5.1 in Barles \cite{AchdouBarlesLitvinovIshii2013} on the subdomain\footnote{Strictly speaking the supersolution $\hat{V}$ could be infinite-valued in part of this domain, but since $\hat{V}$ is continuous in the sense of Definition~\ref{Defninfinitesupersolns}, this requires only a trivial extension of the theorem, with essentially no change to its proof.} $E \times (0,T_0)$ with the subsolution $u$ and the supersolution $\hat{V}$, to deduce that $u \leq \hat{V}$ in $E \times (0,T_0)$. By \eqref{eq:vMequalhatVonK}, we have that $u \leq v^M$ in $K \times (0,T_0)$.

Since any point $(x,t) \in Q^M_T$ belongs to a set of the form $K \times (0,T_0)$ for some compact $K \subset U$ and some $T_0 \in (0,T]$, and since $v^M(x,t) = \infty$ for $(x,t) \in Q_T \setminus Q^M_T$, we have the result.
\end{proof}

\begin{corollary}\label{CorovissubcompHJB}
Let $u$ be a viscosity subsolution of the HJB equation \eqref{eq:HJB} in $Q_T$, for some $T > 0$, and suppose that $u(x,0) \leq v_0(x)$ for all $x \in U$. Then $u \leq v$ in $Q_T$.
\end{corollary}

\begin{proof}
Fix an arbitrary point $(\ox,\ot) \in Q_T$, and let $\varepsilon > 0$. Then there exists a control $(\overline{\alpha},\overline{\beta}) \in \cA$ such that $J(\ox,\ot;\overline{\alpha},\overline{\beta}) < v(\ox,\ot) + \varepsilon$.

As we are only interested in the value of this control on the compact time interval $[0,\ot]$, we may assume without loss of generality that it is uniformly bounded, so that there exists an $M > 0$ such that $(\overline{\alpha},\overline{\beta}) \in \cA^M$.

Since $u$ is a viscosity subsolution of the HJB equation \eqref{eq:HJB}, and $H_M \leq H$, it follows that $u$ is also a viscosity subsolution of the modified HJB equation \eqref{eq:HJBM} in $Q_T$. By Proposition~\ref{PropvissubcompHJBM}, we have that $u \leq v^M$ in $Q_T$. In particular, since $(\overline{\alpha},\overline{\beta}) \in \cA^M$, we have that
$$u(\ox,\ot) \leq v^M(\ox,\ot) \leq J(\ox,\ot;\overline{\alpha},\overline{\beta}) < v(\ox,\ot) + \varepsilon.$$
Letting $\varepsilon \to 0$, we deduce the result.
\end{proof}

\begin{proposition}\label{PropvissupercompHJB}
Suppose that $\hu$ is a viscosity supersolution of the HJB equation \eqref{eq:HJB} in $Q$. Assume that $\hu \in \cVt$, and that $v_0(x) \leq \hu(x,0)$ for all $x \in U$. Then $v \leq \hu$ in $Q$.
\end{proposition}

\begin{proof}
Fix a $T > 0$ and an arbitrary compact set $K \subset U$. We can argue, exactly as in the proof of Proposition~\ref{PropvissubcompHJBM}, that there exists a constant $C_0 > 0$ such that, when considering the value of $v(x,t)$ for $x \in K$ and $t \in (0,T)$, it is sufficient to consider only those controls $(\alpha,\beta)$ which satisfy
\begin{equation}\label{eq:boundonv0again}
v_0\big(w(0;x,t,\alpha,\beta)\big) < C_0.
\end{equation}

Since $v_0 \in \cV$, there exists $\delta > 0$ sufficiently small such that $K \subset U_\delta$, and such that $v_0 \geq C_0$ in $U \setminus U_\delta$, where $U_\delta := \{x = (x_1,x_2) \in U\, :\, |x| < 1/\delta,\ x_2 > \delta\}$. For this $\delta$, let $V_0$ be a finite-valued, locally Lipschitz function on $\R^2$ such that $V_0 = v_0$ in $U_\delta$, and $V_0 \geq C_0$ in $\R^2 \setminus U_\delta$, and define
\begin{equation}\label{eq:Vdefnforsuper}
V(x,t) := \inf_{(\alpha,\beta) \in \cA}\bigg\{\int_0^t\gamma(s,\alpha_s,\beta_s)\,\rd s + V_0\big(w(0;x,t,\alpha,\beta)\big)\bigg\}
\end{equation}
for $x \in U$ and $t \geq 0$. As usual, we associate an infinite `initial' cost to state trajectories which `blow up' after a finite time. The function $V$ is defined similarly to the value function $v$, except with initial cost function $V_0$ instead of $v_0$. Note however that, since $V_0$ is finite-valued in the entirety of $\R^2$, the state trajectories are allowed to leave the domain $U$. It is for this reason that we explicitly defined the function $f$ for values of $x \in \R^2 \setminus U$.

We can argue, exactly as in the first part of the proof of Proposition~\ref{valuefuncvissolnHJB}, that $V$ is a viscosity subsolution of the HJB equation \eqref{eq:HJB} in $Q$, with initial condition $V(x,0) = V_0(x)$ for all $x \in U$.

Since $v_0 \geq C_0$ and $V_0 \geq C_0$ in $\R^2 \setminus U_\delta$, it follows from \eqref{eq:boundonv0again} that when considering the values of $v(x,t)$ and $V(x,t)$ for $x \in K$ and $t \in (0,T)$, we may ignore all controls $(\alpha,\beta)$ such that the corresponding initial point $w(0;x,t,\alpha,\beta)$ lies in $\R^2 \setminus U_\delta$. Since, $v_0$ and $V_0$ are equal in $U_\delta$, we deduce that
\begin{equation}\label{eq:vequalVonK}
v(x,t) = V(x,t) \qquad \text{for all} \quad\ x \in K,\ t \in (0,T).
\end{equation}

Heuristically, if $V_0$ grows slowly as $|x| \to \infty$, then $V$ should also grow slowly as $|x| \to \infty$. In particular, since $\hu \in \cVt$, if $V_0$ grows sufficiently slowly, then we should be able to ensure that $V$ grows slower than $\hu$, so that $\hu$ dominates $V$ asymptotically. We will now make this more precise.

Let us write $(\alpha^0,\beta^0)$ for the control satisfying $(\alpha_s,\beta_s) = (0,0)$ for all $s \geq 0$, and $w^0(\cdot) = w(\cdot\,;x,t,\alpha^0,\beta^0)$ for the corresponding state trajectory. Recall the function $f$ which defines the ODE satisfied by the state trajectories, given by \eqref{eq:defnf} and \eqref{eq:defnfx2neg}. In this case we have
$$\frac{\rd w^0}{\rd s}(s) = f(w^0(s),s,0,0) = \bigg(\begin{array}{c}
0\\
c_s^2
\end{array}\bigg)$$
so that
$$w^0(0) = x - \bigg(\begin{array}{c}
0\\
\int_0^tc_s^2\,\rd s
\end{array}\bigg).$$
In other words, for any point $(x,t) \in Q_T$, there exists a control such that the corresponding state trajectory $w(\cdot)$ `doesn't move very far'.

Recalling Assumption~\ref{assumptiongammafinite}, since
$$V(x,t) \leq V_0(w^0(0)) + T\sup_{0 \leq s \leq T}\gamma(s,0,0),$$
and $\hu \in \cVt$, we deduce that the function $V_0$ may be chosen so that $V_0(x) \to \infty$ as $|x| \to \infty$, but such that this convergence to infinity is sufficiently slow to ensure that $V_0(x) \leq v_0(x)$ for all $x \in U$, and also that
$$\sup_{0 \leq t \leq T}V(x,t) \leq \inf_{0 \leq t \leq T}\hat{u}(x,t)$$
for all sufficiently large values of $|x|$ (with $x_2 > 0$), and for all sufficiently small values of $x_2$. Thus, there exists a bounded open subset $E \Subset U$, such that $K \subset E$, and such that $V \leq \hat{u}$ on $\pa E \times [0,T]$.

Since $V_0$ is locally Lipschitz continuous and $V_0(x) \to \infty$ as $|x| \to \infty$, the result of Theorem~\ref{valuefncloclip} also applies to $V$, and implies that it is also locally Lipschitz continuous. The only addition to the proof of Theorem~\ref{valuefncloclip} in this case is to check that the estimate \eqref{eq:logbound0r} still holds when the trajectories are in $\R^2 \setminus U$; however this is easy, as $f$ is simply given by \eqref{eq:defnfx2neg} for $x \in \R^2 \setminus U$.

As $V(x,0) = V_0(x) \leq v_0(x) \leq \hat{u}(x,0)$ for all $x \in U$, we see that $V \leq \hat{u}$ on the parabolic boundary of $E \times (0,T)$. Since $V$ is Lipschitz continuous on $\overline{E} \times [0,T]$, we may apply Theorem~5.1 in Barles \cite{AchdouBarlesLitvinovIshii2013} on the subdomain $E \times (0,T)$ with the subsolution $V$ and the supersolution $\hat{u}$, to deduce that $V \leq \hat{u}$ in $E \times (0,T)$. By \eqref{eq:vequalVonK} we have that $v \leq \hat{u}$ in $K \times (0,T)$. Since $K$ and $T$ were arbitrary, we have the result.
\end{proof}

Theorem~\ref{TheoremUniqueness} now follows by combining the results of Corollary~\ref{CorovissubcompHJB} and Proposition~\ref{PropvissupercompHJB}.

\begin{remark}[The multivariate case]\label{RemarkMultivariate}
For simplicity of presentation we have assumed that the signal process $X$ and observation process $Y$ are one-dimensional. This led us to an HJB equation where the spatial variable $x$ corresponds (via a change of variables) to the mean and variance of the conditional distribution of the signal, and consequently lives in a half-plane.

We now remark that all of our analysis follows analogously in the multivariate case, where $X$ and $Y$ take values in, say, $\R^d$ and $\R^m$ respectively. In this case $x_1$ takes values in $\R^d$, and $x_2$ takes values in the space of positive definite symmetric $d \times d$ matrices. This spatial domain does not have such a straightforward geometric interpretation, but it is still perfectly valid to consider existence and uniqueness of solutions of the corresponding HJB equation on this domain.

One should note that the derivation of the key estimate in Lemma~\ref{lemmakeyest} is still valid in the multivariate case, as we can still drop the highest order terms in $w$, using the fact that $w_2(s)$ is positive definite and symmetric, and $\beta_s$ is now a positive semi-definite symmetric matrix.

One limitation of this approach is that our change of variables requires the conditional covariance matrix $R$ to be invertible, so assuming it to be merely positive semi-definite is not enough. The initial cost function $v_0(w(0))$ should therefore be infinite-valued whenever $w_2(0)$ is not positive definite, and satisfy a coercivity condition analogous to Definition~\ref{defncV}.
\end{remark}

\section{Numerical example}\label{Sectionnumericalexample}

We conclude with a numerical example. We simulate a realisation of the signal process $X$ with `true' parameters as in the table (where for simplicity $\alpha$ and $\beta$ are taken to be constant in time), and the observation process $Y$ with $c = 1$. We then consider the problem of estimating $X$ using the `estimated' parameters as shown.

\begin{center}
\begin{tabular}{ |c|c|c|c|c| }
\hline
true parameters & $\alpha = 0.5$ & $\beta = 1.5$ & $\mu_0 = 1$ & $\sigma_0 = 0.2$ \\
\hline
estimated parameters & $\alpha^\ast = 0$ & $\beta^\ast = 1$ & $\mu_0^\ast = 0$ & $\sigma_0^\ast = 1$ \\
\hline
\end{tabular}
\end{center}
We use the penalty functions given by
\begin{gather}
\gamma(t,a,b) = 5(a - \alpha^\ast)^2 + 10(b - \beta^\ast)^2,\label{eq:gammaexample}\\
v_0(x_1,x_2) = 15(x_1 - x_1^\ast)^2 + 15(x_2 - x_2^\ast)^2,\label{eq:v0example}
\end{gather}
where $x_1^\ast = \frac{\mu_0^\ast}{(\sigma_0^\ast)^2}$ and $x_2^\ast = \frac{1}{(\sigma_0^\ast)^2}$.

After solving the HJB equation \eqref{eq:HJB} for the value function $v$ (see the Appendix for the numerical scheme used in the current example), we reverse the change of variables performed in Section~\ref{SectionReformulation} to obtain the penalty function $\kappa$. We are then in a position to calculate the nonlinear expectation via \eqref{eq:NLErelatestokappa}. In this example we choose $k_1 = 10$ and $k_2 = 5$.

In Figure~\ref{fig:estimate_X} we compare a robust minimax estimate of the signal, calculated as $\argmin_{\xi} \cE\big((X_t - \xi)^2\,\big|\,\cY_t\big)$, with the standard Kalman--Bucy filter, itself calculated with both the esimated parameters ($\E^{\text{est}}$) and with the true parameters ($\E^{\text{true}}$). The latter is of course only computable with knowledge of the true parameters, and theoretically represents the `best' estimate one could hope to achieve.

\begin{figure}[!ht]
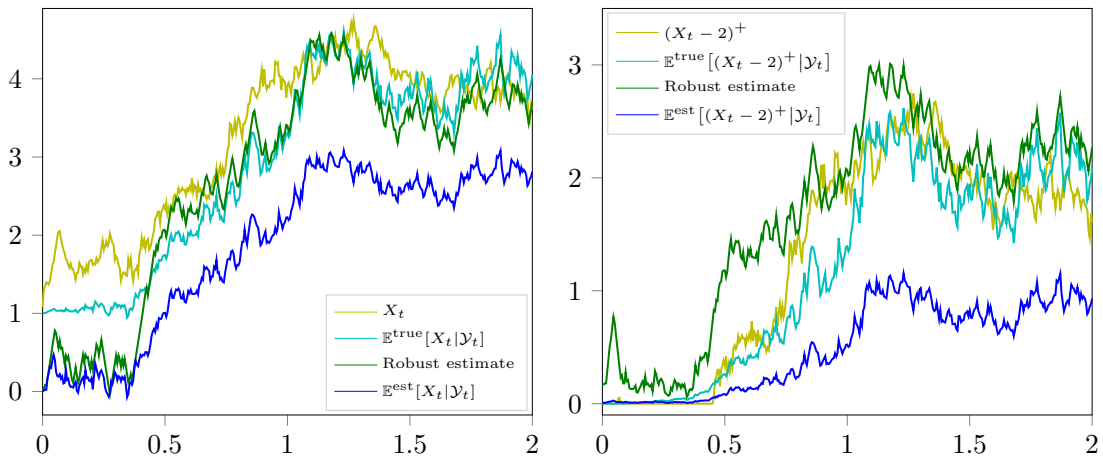

\centering
\begin{subfigure}{.49\textwidth}
  \centering
  \input{estimate_X}
  \caption{Estimates of $X_t$.}
  \label{fig:estimate_X}
\end{subfigure}
\hfill
\begin{subfigure}{.49\textwidth}
  \centering
  \input{estimate_X-2+}
  \caption{Estimates of $(X - 2)^+$.}
  \label{fig:estimate_(X-2)+}
\end{subfigure}
\caption{Comparison of standard and robust filter estimates.}
\end{figure}

As mentioned in the introduction, the results of the majority of the existing literature on robust filtering allow for estimation only of the state of the signal, or linear functions thereof. A significant strength of the current approach is that it also enables one to compute estimates of nonlinear functions of the signal. Moreover, this can be done using the same penalty function---that is, we do not have to solve the PDE again in order to perform each subsequent estimate. To illustrate this, in Figure~\ref{fig:estimate_(X-2)+} we compute a robust estimate of the variable $(X_t - 2)^+$\hspace{-2pt}, i.e.~$\max(X_t - 2,0)$, as $\argmin_{\xi}\cE\big(\big((X_t - 2)^+ - \xi\big)^{\hspace{-1pt}2}\,\big|\,\cY_t\big)$, and compare it with standard filter estimates.

We also calculate the `upper' and `lower' expectations of the signal, i.e.~$\cE(X_t\,|\,\cY_t)$ and $-\cE(-X_t\,|\,\cY_t)$ respectively, which can be thought of as providing error bounds for the Kalman--Bucy filter, and are shown in Figure~\ref{fig:upper_lower_X}. Finally, Figure~\ref{fig:upper_lower_(X-2)+} shows the upper and lower expectations of $(X_t - 2)^+$\hspace{-2pt}. We note in particular the asymmetry between the upper and lower expectations around the Kalman--Bucy filter, in contrast to classical equal-sided confidence intervals.

\begin{figure}[!ht]
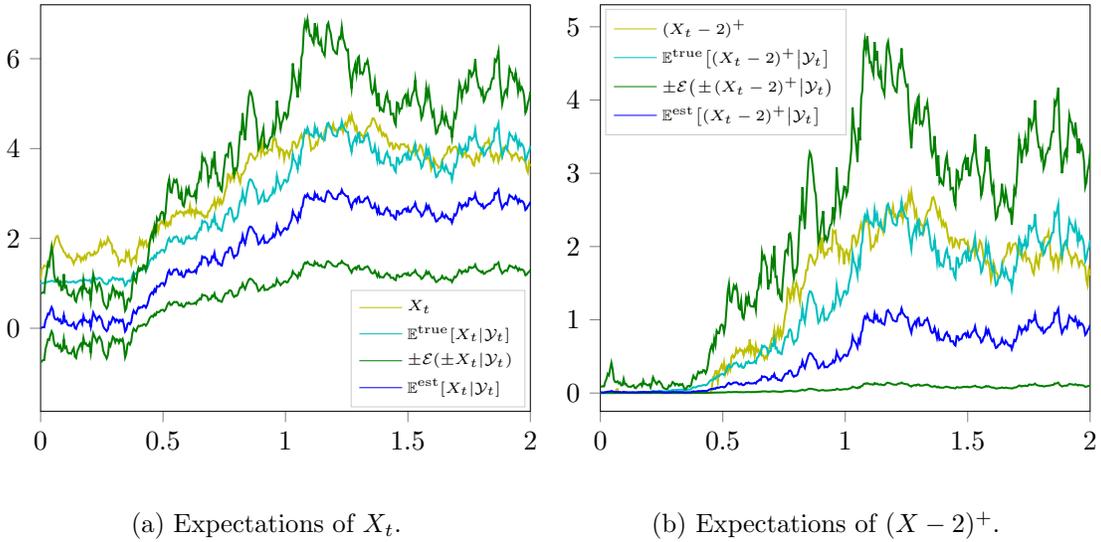

\centering
\begin{subfigure}{.49\textwidth}
  \centering
  \input{upper_lower_X}
  \caption{Expectations of $X_t$.}
  \label{fig:upper_lower_X}
\end{subfigure}
\hfill
\begin{subfigure}{.49\textwidth}
  \centering
  \input{upper_lower_X-2+}
  \caption{Expectations of $(X - 2)^+$.}
  \label{fig:upper_lower_(X-2)+}
\end{subfigure}
\caption{Robust upper and lower expectations.}
\end{figure}

\section*{Appendix -- A numerical scheme}

We briefly outline a numerical scheme to compute the solution of the HJB equation \eqref{eq:HJB}, which was used to obtain the robust filter estimates exhibited in Section~\ref{Sectionnumericalexample}.

As the `reference' parameters $\alpha^\ast, \beta^\ast, \mu_0^\ast, \sigma_0^\ast$ are a priori considered to be the best estimates for the true parameters, we set their penalty to zero. That is, our penalty functions $\gamma$ and $v_0$ will naturally be equal to zero for this choice of parameters, as in \eqref{eq:gammaexample}--\eqref{eq:v0example}. Thus, recalling the change of variables in \eqref{eq:changevariablesmusigmaxt}, and the definition of the value function $v$ in \eqref{eq:valuefuncdefn}, we note that the transformed reference dynamics, given by
$$w^\ast(t) = \bigg(\frac{q^\ast_t}{R^\ast_t} - \eta_t,\,\frac{1}{R^\ast_t}\bigg),$$
where $q^\ast$ and $R^\ast$ are the conditional mean and variance calculated using the reference parameters, traces out the minimum point of the value function. That is, we have that $v(w^\ast(t),t) = 0$ for every $t \geq 0$. Since in practice we are only interested in the local behaviour of $v$ around its minimum, we only need to solve the PDE in a finite subdomain centred at this (moving) minimum point.

Defining $\bar{v}(\zeta,t) = v(w^\ast(t) + \zeta,t)$, the HJB equation \eqref{eq:HJB} becomes
\begin{equation*}
\frac{\pa\bar{v}}{\pa t}(\zeta,t) + \sup_{(a,b) \in \R \times [0,\infty)}\big\{\bar{f}(\zeta,t,a,b)\cdot\nabla\bar{v}(\zeta,t) - \gamma(t,a,b)\big\} = 0,
\end{equation*}
where
\begin{equation*}
\bar{f}(\zeta,t,a,b) = f(w^\ast(t) + \zeta,t,a,b) - f(w^\ast(t),t,\alpha^\ast_t,\beta^\ast_t).
\end{equation*}

Since the value function $v$ is unbounded (indeed, it `blows up' at the boundary in the sense of Definition~\ref{defnVtloc}) we instead solve for the function $\lambda = -1/(1 + \bar{v})$, which takes values in the bounded interval $[-1,0)$, and satisfies the equation
\begin{equation}\label{eq:HJBlambda}
\frac{\pa\lambda}{\pa t}(\zeta,t) + \sup_{(a,b) \in \R \times [0,\infty)}\Big\{\bar{f}(\zeta,t,a,b)\cdot\nabla\lambda(\zeta,t) - \lambda(\zeta,t)^2\gamma(t,a,b)\Big\} = 0,
\end{equation}
with the boundary condition that $\lambda(\zeta,t) \to 0$ as $|\zeta| \to \infty$ and as $\zeta_2 \searrow -w^\ast(t)_2$. Since the transformation from $\bar{v}$ to $\lambda$ is monotone increasing, the corresponding viscosity theory carries over.

We approximate the solution of \eqref{eq:HJBlambda} with an explicit finite difference scheme on a finite rectangular domain, centred at $\zeta = 0$, on the boundary of which we impose a zero boundary condition.

\begin{remark}
Although the function $\lambda$ is not actually equal to zero on the boundary of a finite subdomain, this procedure does lead to a convergent scheme as we increase the size of the domain. This can be seen by recalling, from the proofs of Propositions~\ref{PropvissubcompHJBM} and \ref{PropvissupercompHJB}, the fact that $v$ (and hence $\lambda$) only depends on the local values of its initial condition, and thus, on any given inner region, is unaffected by perturbations in any sufficiently large outer region, provided that such perturbations only increase the value of $v$ (or $\lambda$).
\end{remark}

By the comparison results obtained in Section~\ref{SectionUniqueness}, since $\lambda$ is bounded, we recover the setting of Barles and Souganidis \cite{BarlesSouganidis1991}. To establish monotonicity we use an upwind scheme, truncating large values of the drift $\bar{f}$ in order to ensure a suitable Courant--Friedrichs--Lewy condition holds. By Theorem~2.1 in \cite{BarlesSouganidis1991}, the scheme is uniformly convergent, and by the previous remark, the entire solution $\lambda$ may be obtained by increasing the size of the domain. The value function $v$ may then be recovered via $v(x,t) = -1 - 1/\lambda(x - w^\ast(t),t)$.

\end{document}